\documentclass[12pt]{amsart}

\usepackage{amssymb}

\usepackage{enumerate}

\usepackage{graphicx}

\makeatletter
\@namedef{subjclassname@2010}{%
  \textup{2010} Mathematics Subject Classification}
\makeatother

\newtheorem{thm}{Theorem}
\newtheorem{cor}[thm]{Corollary}
\newtheorem{lem}[thm]{Lemma}

\theoremstyle{definition}

\begin{document}


\baselineskip=17pt



\title{On the number of solutions of decomposable form inequalities}

\author[C.L.~Stewart]{C.L.~Stewart}
\address{Department of Pure Mathematics, University of Waterloo \\
Waterloo, Ontario, Canada N2L 3G1}
\email{cstewart@uwaterloo.ca}

\dedicatory{}

\date{}

\begin{abstract}
 In 2001 Thunder gave an estimate for the number of integer solutions of decomposable form inequalities under the assumption that the forms are of finite type. The purpose of this article is to generalize this result to forms which are of essentially finite type. In the special case of binary forms this gives an improvement of a result of Mahler from 1933.
\end{abstract}

\subjclass[2020]{Primary 11D75, 11J87; Secondary 11D57}

\keywords{decomposable forms, Subspace Theorem}

\maketitle

\section{Introduction}
Let $n$ be an integer with $n \geq 2$ and put $\textbf X$$ = (X_1,...,X_n)$. Let $F$ be a non-zero decomposable form in $n$ variables with integer coefficients and degree $d$ with 
$ d> n$, so

\begin{equation} \label{eq1}
F(\textbf{X}) = L_1(\textbf {X})....L_d(\textbf {X})
\end{equation}
where $L_1(\textbf{X}), ... ,L_d(\textbf{X})$ are linear forms in $\mathbb C[X_1,...,X_n]$.  Let $m$ be a positive integer and let $N_F(m)$ denote the number of points $(a_1,...,a_n)$ with integer coordinates for which
\begin{equation} \label{eq2}
|F(a_1,...,a_n)| \leq m.
\end{equation}
Let $V_F$ denote the volume of the set
$$
\{(x_1,...,x_n) \in \mathbb R^n : |F(x_1, ... ,x_n)| \leq 1 \}.
$$
By homogeneity the volume of
\begin{equation} \label{eq3}
\{(x_1,...,x_n) \in \mathbb R^n : |F(x_1, ... ,x_n)| \leq m \}
\end{equation}
is $V_Fm^{n/d}$ and one might suppose that $N_F(m)$ is close to $V_Fm^{n/d}$.

$F$ is said to be of \emph{finite type} if $V_F$ is finite and the same is true for $F$ restricted to any non-trivial rational subspace. In particular, for every $n'$-dimensional subspace $S$ of $\mathbb{R}^n$ defined over $\mathbb{Q}$ the $n'$-dimensional volume of $F$ restricted to $S$ is finite. In 2001 Thunder \cite{Th1} showed that if $F$ is of finite type then
\begin{equation} \label{eq4}
N_F(m) \ll_{n,d} m^{n/d} ;
\end{equation}
throughout this paper the symbol $\ll $ together with a subscript will mean less than a positive number which depends on the terms in the subscript. Thunder's result resolved a conjecture of Schmidt \cite{Sch3} and is best possible up to the dependence of the implicit constant on $n$ and $d$.

For any element $\textbf x=(x_1,..., x_n)$ in $\mathbb{C}^n$ let $\left\Vert \textbf x \right\Vert = (x_1\overline x_1 + ... +x_n\overline x_n)^{1/2} $. For any linear form $L(\textbf{X})= \alpha_1X_1+ ... + \alpha_nX_n$ in $\mathbb{C}[X_1, ... ,X_n]$ let $\textbf{L}$ denote the coefficient vector $(\alpha_1, ... ,\alpha_n)$ of $L(\textbf{X})$. We define the quantity $\mathcal{H}(F)$ of $F$ by
$$
\mathcal{H}(F) = \prod^d_{\substack{i=1 }} \left\Vert \textbf L_i \right\Vert .
$$

Thunder \cite{Th1}, \cite{Th3} also proved that if $F$ is of finite type and $F$ is not proportional to a power of a definite quadratic form in $2$ variables then there exist positive numbers $a_F$ and $c_F$ such that
\begin{equation} \label{eq5}
|N_F(m) - m^{n/d}V_F|  \ll_{n,d}  \mathcal{H}(F)^{c_F}(1+ \log m)^{n-2}m^{\frac{n-1}{d-a_F}}.
\end{equation}
If the discriminant of the form is non-zero then one may take $a_F=1$ and $c_F= \binom {d-1}{n-1} -1.$

If $T$ is in $\operatorname{GL}_n(\mathbb{Z})$ then the form $G(\textbf{X})= F(T(\textbf{X}))$ is said to be equivalent to $F$. Then $V_F=V_G$ but $\mathcal{H}(F)$ need not be equal to $\mathcal{H}(G)$. Put
$$
\mathcal{H}_0(F) =  \min_T \mathcal{H}(F\circ T)
$$
where the minimum is taken over $T$ in $\operatorname{GL}_n(\mathbb{Z})$. Thunder \cite{Th2} showed that if $F$ is of finite type and $F$ is not proportional to a power of a definite quadratic form in $2$ variables then
$$
V_F \ll_{n,d} \mathcal{H}_0(F)^{-1/d}(1 + \log \mathcal{H}_0(F) )^{n-1}. 
$$

In 1933 Mahler \cite{M} proved that if $n=2$ and $F(X_1, X_2)$ is a binary form with integer coefficients which is irreducible over the rationals then
\begin{equation} \label{eq6}
|N_F(m) - m^{2/d}V_F|  \ll_{F}  m^{1/(d-1)}.
\end{equation}
Thunder's result \eqref{eq5} is a generalization of \eqref{eq6} since if $F$ is irreducible over $\mathbb{Q}$ then $a_F=1$ and $F$ is of finite type.

Ramachandra, in 1969 \cite{R}, was the first to obtain an asymptotic result for $N_F(m)$ for a class of decomposable forms with $n \geq 3$. He did so when $F$ has the shape
$$
F(\textbf{X}) = N_{\mathbb{K}/\mathbb{Q}}(X_1 + \alpha X_2 + \alpha^2X_3 + ... + \alpha^{n-1}X_n)
$$
where $\mathbb{K} = \mathbb{Q}(\alpha)$ is a number field of degree $r$ with $r \geq 8n^6$ and $N_{\mathbb{K}/\mathbb{Q}}$ denotes the norm from $\mathbb{K}$ to $\mathbb{Q}.$

Let $\alpha_1, ... ,\alpha_n$ be non-zero algebraic numbers and put $\mathbb{K} = \mathbb{Q}(\alpha_1, ... ,\alpha_n)$. Suppose that $F(\textbf{X})$ is a norm form so
$$
F(\textbf{X}) = N_{\mathbb{K}/\mathbb{Q}}(\alpha_1X_1 + ... +\alpha_nX_n) = \prod_{\sigma}\sigma(\alpha_1X_1 + ... +\alpha_nX_n)
$$
where the product is taken over the isomorphic embeddings $\sigma$ of $\mathbb{K}$ into $\mathbb{C}$. Let $V$ be the vector space of all rational linear combinations of $\alpha_1, ... ,\alpha_n$. For each subfield $\mathbb{J}$ of $\mathbb{K}$ we define the linear subspace $V^{\mathbb{J}}$ of $V$ given by the elements of $V$ which remain in $V$ after multiplication by any element of $\mathbb{J}$. $F$ is said to be non-degenerate if $\alpha_1, ... ,\alpha_n$ are linearly independent over $\mathbb{Q}$ and if $V^{\mathbb{J}} = \{0\}$ for each subfield $\mathbb{J}$ of $\mathbb{K}$ which is not $\mathbb{Q}$ or an imaginary quadratic field. In 1972 Schmidt \cite{Sch1} proved that $N_F(m)$ is finite for each positive integer $m$ if and only if $F$ is non-degenerate; see \cite{JH1} and Chapter 9 of \cite{EG} for quantitative results for decomposable form equations. In 2000 Evertse \cite{JH3} proved that if $F$ is a non-degenerate norm form then
\begin{equation} \label{eq7}
N_F(m) \leq (16d)^{(n+1)^3/3}(1+ \log m)^{n(n-1)/2}m^{(n+ \sum^{n-1}_{m=2}1/m)/d)}.
\end{equation}
Non-degenerate norm forms are of finite type and so \eqref{eq4} gives a better dependence on $m$ than \eqref{eq7} although the dependence of the upper bound on $n$ and $d$ is not explicit in \eqref{eq4}.

Let $N^*_F(m)$ denote the number of vectors $(a_1, ... ,a_n)$ with integer coordinates for which
\begin{equation} \label{eq8}
0 < |F(a_1, ... ,a_n)| \leq m.
\end{equation}
If $F$ is of finite type then $F$ does not vanish at any non-zero integer point and so
\begin{equation} \label{eq9}
N_F(m) = 1 + N^*_F(m).
\end{equation}

There exist distinct irreducible polynomials $F_1, ... ,F_k$ with integer coefficients, content $1$ and degrees $d_1, ... ,d_k$ respectively and there exist positive integers $l_1, ... ,l_k$ for which $d_1l_1+ ... +d_kl_k = d$ such that
\begin{equation} \label{eq10}
F(\textbf{X})=C_0F_1(\textbf{X})^{l_1} ... F_k(\textbf{X})^{l_k},
\end{equation}
where $|C_0|$ is the content of $F$. By, for instance, the discussion in Section 2, for each integer $j$ with $1\leq j \leq k$ the polynomial $F_j(\textbf{X})$ is of the form $aN_{\mathbb{K}/\mathbb{Q}}(L(\textbf{X}))$ where $a$ is a non-zero rational number, $\mathbb{K}$ is a number field of degree $d_j$ over $\mathbb{Q}$,  $N_{\mathbb{K}/\mathbb{Q}}$ denotes the norm from $\mathbb{K}$ to $\mathbb{Q}$ and $L(\textbf{X})$ is a linear form which is proportional to a linear form $L_i$ with $i$ from $\{1, ... ,d\}$.

 For $i= 1, ... ,d$ let $B_i$ be the rational subspace of $\mathbb{R}^n$ for which $L_i(\textbf{X})=0$. Note that if $L_i(\textbf{X})$ and $L_j(\textbf{X})$ divide $F_h(\textbf{X})$ in $\mathbb{C}[\textbf{X}]$ for some $h$ with $1\leq h\leq k$ then $B_i=B_j$. Thus each polynomial $F_i(\textbf{X})$ determines exactly one rational subspace of $\mathbb{R}^n$, say $A_i$, for which $F_i(\textbf{X})=0$. Put
\begin{equation} \label{eq11}
d_F = \begin{cases}
0 & \text{if}\ A_i= \{\textbf{0}\}\ \text{for}\ \ i=1, ... ,k \\
\max \{l_{i_1}d_{i_1} + ... +l_{i_j}d_{i_j}\} & \text{otherwise},
\end{cases}
\end{equation}
where the maximum is taken over those tuples $(i_1, ... ,i_j)$ of distinct integers for which $A_{i_1}\cap ...\cap A_{i_j}$ is different from the zero vector or equivalently for which there is a non-zero integer point $(s_1, ... ,s_n)$ for which $F_{i_m}(s_1, ... ,s_n)=0$ for $m=1, ... ,j$.

$F$ is said to be of \emph{essentially finite type} if $V_F$ is finite, $V(\tilde F) $ is finite whenever $\tilde F$ is $F$ restricted to a rational subspace of $\mathbb{R}^n$ which is not a subspace of $A_i$ for $i=1, ... ,k$ and
\begin{equation} \label{eq12}
A_1\cap ... \cap A_k = \{\textbf{0} \}.
\end{equation}

If $F$ is of essentially finite type then, by virtue of \eqref{eq12} ,
\begin{equation} \label{eq13}
d_F < d.
\end{equation}
Further, if $F$ is of finite type then it is also of essentially finite type since in this case $A_i=\{ \textbf{0}\}$ for $i=1, ... ,k$ and so \eqref{eq12} holds.

\begin{thm} \label{Theorem 1}
Let $F(\textbf{X})$ be a non-zero decomposable form in $n$ variables with integer coefficients and degree $d$ with $d>n\geq2$ and let $m$ be a positive integer. If $F$ is of essentially finite type then
\begin{equation} \label{eq14}
N^*_F(m) \ll_{n,d} m^{\frac{1}{d} + \frac{n-1}{d-d_F}}.
\end{equation}
 
\end{thm}
Notice that if $F$ is of finite type then $d_F=0$ and \eqref{eq4} follows from \eqref{eq9} and \eqref{eq14}.

The proof of Theorem \ref{Theorem 1} depends on a quantitative version of Schmidt's Subspace Theorem due to Evertse \cite{JH2}. A key feature of Theorem 1 is that the upper bound for $N^*_F(m)$ is independent of the coefficients of the form $F$. We require such an estimate in order to prove the analogue of estimate \eqref{eq5} for forms of essentially finite type. Before stating such a result we shall make explicit the quantities $a_F$ and $c_F$.

Any linear form $L_i(\textbf{X})$ in the decomposition of $F(\textbf{X})$ as in \eqref{eq1} is a factor of $F_j(\textbf{X})$ for some $j$ with $1\leq j \leq k$ by \eqref{eq10} and so is proportional to a linear form with coefficients in a number field of degree the degree of $F_j(\textbf{X})$. For a factorization as in \eqref{eq1} of $F$ we let $I(F)$ denote the set of all  $n$-tuples $(\textbf{L}_{i_1}, ... ,\textbf{L}_{i_n})$ of linearly independent coefficient vectors. For each linear form $L_i(\textbf{X})$ from \eqref{eq1} we denote by $b(L_i)$ the number of $n$-tuples in $I(F)$ which contain $\textbf{L}_i$ and we put
$$
b_F= \max \{b(L_1), ... ,b(L_n)\}. 
$$
Next let $J(F)$ be the subset of $I(F)$ consisting of $n$-tuples $(\textbf{L}_{i_1}, ... ,\textbf{L}_{i_n})$ for which for $j=1, ... ,n-1$ either $\textbf{L}_{i_{j+1}}$ is proportional to $\overline{\textbf{L}}_{i_j}$ or $\overline {\textbf{L}}_{i_j}$ is in the span of $\textbf{L}_{i_1}, ... ,\textbf{L}_{i_j}$. Thunder \cite{Th1} showed that $J(F)$ is non-empty provided that $I(F)$ is non-empty. We then put
$$
a_F = \max \left \{ \frac { \text{the number of }  \textbf{L}_i  \text{ in the span of }  \textbf{L}_{i_1}, ... ,\textbf{L}_{i_j}}{j}\right \}
$$
where the maximum is taken over integers $j$ from $\{1, ... ,n-1\}$ and $n$-tuples $(\textbf{L}_{i_1}, ... ,\textbf{L}_{i_n})$ from $J(F)$. Thunder proved, see (16) of \cite{Th1}, that if $F$ is of finite type and $F$ is not proportional to a power of a definite quadratic form in $2$ variables then
\begin{equation} \label{eq15}
1 \leq a_F \leq \frac{d}{n} - \frac{1}{n(n-1)} ,
\end{equation}
and the same argument shows that \eqref{eq15} holds if $F$ is of essentially finite type and $F$ is not proportional to a power of a definite quadratic form in $2$ variables. We note that $a_F=1$ if and only if the discriminant $\Delta_F$ of $F$ is non-zero.

Finally we put
$$
c_F = \begin{cases}
\binom{d-1}{n-1} -1 & \text{if}\ \Delta_F \neq 0 \\
\frac{b_F}{n!a_F}(d-(n-1)a_F)- \frac{1}{a_F} & \text{otherwise}.
\end{cases}
$$

For any set $X$ let $|X|$ denote its cardinality. Let $I'(F)$ be the subset of $I(F)$ consisting of the $n$-tuples $(\textbf{L}_{i_1}, ... , \textbf{L}_{i_n})$ of linearly independent coefficient vectors with $i_1 < i_2 < ... < i_n$. Then
$$
|I'(F)| \leq \binom {d}{n} .
$$
Since $b_F \leq n!|I'(F)|$ and $a_F \geq 1$ we find that if $\Delta_F=0$ then
$$
c_F \leq \binom {d}{n}\left(\frac{d}{a_F} - (n-1)\right)
$$
hence
\begin{equation} \label{eq16}
c_F \leq \binom{d}{n} (d-n+1).
\end{equation}
Certainly $\binom{d-1}{n-1} < \binom{d}{n}$ when $d > n$ and so \eqref{eq16} also holds when $\Delta_F \neq 0$. Further $\frac{b_F}{n!} \geq 1$ from which it follows that when $\Delta_F = 0$
$$
c_F \geq \frac{d-(n-1)a_F -1}{a_F}
$$
and so, by \eqref{eq15},
$$
c_F \geq \frac{(n-1)(d-n+1)}{d(n-1)-1}
$$
hence
\begin{equation} \label{eq17}
c_F \geq \frac{d-n+1}{d}.
\end{equation}
Note that \eqref{eq17} also holds when $\Delta_F \neq 0$.

\begin{thm} \label{Theorem 2}
Let $F(\textbf{X})$ be a decomposable form in $n$ variables with integer coefficients and degree $d$ with $d>n\geq2$ and let $m$ be an integer with $m > 1$. If $F$ is of essentially finite type then
\begin{equation} \label{eq18}
|N^*_F(m) - m^{n/d}V_F| \ll_{n,d}\mathcal{H}(F)^{c_F}(\log m)^{n-2}m^{\frac{n-1}{d-a_F}} + (\log m + \log \mathcal{H}(F))^{n-1}m^{\frac{1}{d} + \frac{n-2}{d-d_F}}
\end{equation}
 
\end{thm}

If $F$ is of finite type then $d_F=0$ and $a_F \geq 1$. Thus
$$
\frac{1}{d} + \frac {n-2}{d-d_F} = \frac{n-1}{d} < \frac{n-1}{d-a_F}
$$
and so \eqref{eq5} follows from Theorem \ref{Theorem 2}.

For the proof we appeal to Theorem \ref{Theorem 1} and, once again, to a quantitative version of the Subspace Theorem.

The discriminant $\Delta_F$ of a form as in \eqref{eq1} is given by
$$
\Delta_F = \prod_{(i_1, ... ,i_n)}det(\textbf{L}^{tr}_{i_1}, ... ,\textbf{L}^{tr}_{i_n})
$$
where the product is taken over all $n$-tuples of distinct integers $(i_1, ... ,i_n)$ with $1 \leq i_j \leq d$ for $j = 1, ... ,n$. Here $\textbf{L}^{tr}$ denotes the transpose of $\textbf{L}$.

Let $B(x,y)$ denote the Beta function, see \cite{Davis}. In 1996 Bean and Thunder \cite{BT} proved that if $\Delta_F \neq 0$ then
\begin{equation} \label{eq19}
|\Delta_F|^{\frac{(d-n)!}{d!}}V_F \leq C_n
\end{equation}
where
$$
C_n = \frac{2}{n}\prod^{n-1}_{k=1}\bigg(B( \frac{1}{n+1}, \frac{k}{n+1}) + B(\frac{n-k}{n+1}, \frac{k}{n+1}) + B(\frac{n-k}{n+1}, \frac{1}{n+1})\bigg);
$$
the case when $n=2$ was established by Bean \cite{B} in 1994. They proved that the upper bound of $C_n$ is sharp in \eqref{eq19} and that $C_n$ grows like a constant times $(2n)^n$. If $\Delta_F$ is non-zero then $a_F=1$ and $c_F= \binom{d-1}{n-1} -1$. Thus by Theorem 2 and \eqref{eq19} we have the following result.

\begin{cor} \label{Corollary 3}
Let $F(\textbf{X})$ be a decomposable form in $n$ variables with integer coefficients and degree $d$ with $d>n\geq2$ and let $m$ be an integer with $m > 1$. If $F$ is of essentially finite type  and $\Delta_F \neq 0$ then
\begin{equation} \label{eq20}
N^*_F(m) \ll_{n,d} m^{n/d}|\Delta_F|^{-\frac{(d-n)!}{n!}} +  \mathcal{H}(F)^{\binom{d-1}{n-1}-1}(\log m)^{n-2}m^{\frac{n-1}{d-1}} + (\log m + \log \mathcal{H}(F))^{n-1}m^{\frac{1}{d} + \frac{n-2}{d-d_F}}
\end{equation}
 
\end{cor}

When $n=2$, $F(\textbf{X})$ is a binary form and if $\Delta_F$ is non-zero then $F$ is of essentially finite type and $d_F$ is either $0$ or $1$. Since $a_F=1$ we obtain our next result.

\begin{cor} \label{Corollary 4}
Let $F(\textbf{X})$ be a binary form with integer coefficients, degree $d$ with $d \geq3$ and $\Delta_F \neq 0$.  Let $m$ be a positive integer.  Then
\begin{equation} \label{eq21}
|N^*_F(m) - m^{2/d}V_F| \ll_{d} m^{\frac{1}{d-1}}\mathcal{H}(F)^{d-2}.
\end{equation}
 
\end{cor}

Corollary \ref{Corollary 4} generalizes Mahler's result \eqref{eq6}, where $F$ is assumed to be irreducible over the rationals, to the case where $F$ has a non-zero discriminant. By \eqref{eq5} such a result holds when $F$ is of finite type but that does not give Corollary \ref{Corollary 4} in the case when $F$ has a linear factor over the rationals. Corollary \ref{Corollary 4} is required in the work of Stewart and Xiao \cite{St1} on the number of integers represented by a binary form with a non-zero discriminant and the number of $k$-free integers represented by such a form \cite{St2} and the author is grateful to Professor Fouvry for pointing this out.

The proofs of Theorems \ref{Theorem 1} and \ref{Theorem 2} build on the work of Thunder \cite{Th1}. He proceeds by establishing an upper bound for each $\textbf{x}$ in $\mathbb{R}^n$ for
$$
\frac{\prod^{n}_{j=1}|L_{i_j}(\textbf{x})|}{|det(\textbf{L}_{i_1}^{tr}, ... ,\textbf{L}_{i_n}^{tr})|}
$$
for some $n$-tuple  $(\textbf{L}_{i_1}, ... ,\textbf{L}_{i_n})$ from $I(F)$. Thunder establishes two such estimates and they are given in Lemma 5 and Lemma 6 of \cite{Th1}. Our main innovation is a modification of Lemma 6 in order to treat the more general situation when $F$ is of essentially finite type.

I would like to thank Professors Jeff Thunder and Stanley Xiao for some helpful comments on this work.

\section{Small products of linear forms}

Let $F(\textbf{X})$ be a decomposable form in $n$ variables with integer coefficients and degree $d$ with $d > n \geq 2$ as in \eqref{eq1}.

\begin{lem} \label{lem1}
If $F(\textbf{X})$  is of essentially finite type and $F$ is not proportional to a power of a definite quadratic form in $2$ variables then there is a positive number $C_1=C_1(n,d)$, which depends on $n$ and $d$, such that for every $\textbf{x}$ in $\mathbb{R}^n$ there is an $n$-tuple $(\textbf{L}_{i_1}, ... ,\textbf{L}_{i_n})$ in $J(F)$ for which
$$
\frac{\prod^{n}_{j=1}|L_{i_j}(\textbf{x})|}{|det(\textbf{L}_{i_1}^{tr}, ... ,\textbf{L}_{i_n}^{tr})|} < C_1\bigg(\frac{|F(\textbf{x})|}{\left\Vert \textbf x \right\Vert ^{d-na_F}}\bigg)^{1/a_F}\mathcal{H}(F)^{c_F}.
$$
\end{lem}
\begin{proof}
This follows from Lemma 5 of \cite{Th1} since if $F$ is of essentially finite type and $F$ is not proportional to a power of a definite form in $2$ variables then, by \eqref{eq15}, $a_F < d/n.$
\end{proof}

Suppose that $R$ is a decomposable form in $n$ variables with integer coefficients and degree $d$ with $d> n \geq 2$ as in \eqref{eq1}. Suppose further that $R$ is irreducible over $\mathbb{Q}$ and has content $1$. Then, by Lemme 1 on p. 85 of \cite{BC}, $R$ is equivalent under $\operatorname{GL}_n(\mathbb{Z})$ to a form for which the coefficient of $X_1^d$, say $A$, is non-zero. Thus we may assume that
$$
R(\textbf{X}) =A\prod_{j=1}^{d}( X_1 + \mu_{2,j}X_2 + ... +\mu_{n,j}X_n)
$$
and, with $X_i=1$ and $X_j=0$ for $i$ from $\{2, ... ,n\}$ and $j=2, ... ,n$ with $j\neq i$,
$$
R(X_1, 0, ... ,0,1,0, ... 0) = A\prod_{j=1}^{d}( X_1 +\mu_{i,j})
$$
is a polynomial with integer coefficients. Therefore there is a linear form $L(\textbf{X})= X_1 + \mu_{2,1}X_2 + ... +\mu_{n,1}X_n$ which divides $R$ with $\mu_{i,1}$ a root of 
$R(X_1, 0, ... ,0,1,0, ... ,0)$ where  the $1$ is in the $i$-th coordinate. Put $\mathbb{K} = \mathbb{Q}(\mu_{2,1}, ... ,\mu_{n,1})$. Then
$$
N_{\mathbb{K}/\mathbb{Q}}(L(\textbf{X})) = \prod_{\sigma}\sigma (L(\textbf{X}))
$$
where the product is taken over the isomorphisms of $\mathbb{K}$ into $\mathbb{C}$. $N_{\mathbb{K}/\mathbb{Q}}(L(\textbf{X}))$ is in $\mathbb{Q}[\textbf{X}]$ and is irreducible, see Th\'eor\`eme 2 of \cite{BC}. Thus, since $N_{\mathbb{K}/\mathbb{Q}}(L(\textbf{X}))$ divides $R$, $R$ is a rational multiple of $N_{\mathbb{K}/\mathbb{Q}}(L(\textbf{X}))$. The coordinates of $A\sigma(L(\textbf{X}))$ are algebraic integers and so $A^dN_{\mathbb{K}/\mathbb{Q}}(L(\textbf{X}))$ is a polynomial with integer coefficients. Since $R$ has content $1$ there exists a non-zero integer $m$ such that
$$
mR = A^d\prod_{\sigma}\sigma(L(\textbf{X})) = \prod_{\sigma}\sigma(AL(\textbf{X}))
$$
hence
\begin{equation} \label{eq22}
R=\frac{1}{m}\prod_{\sigma}\sigma(AL(\textbf{X})).
\end{equation}
In particular
\begin{equation} \label{eq22a}
R=\prod_{\sigma}M_{\sigma}(\textbf{X})
\end{equation}
where $M_{id}(\textbf{X}) = \frac{1}{m}AL(\textbf{X})$ and $M_{\sigma}(\textbf{X}) = \sigma(AL(\textbf{X}))$ for $\sigma$ different from the identity.

Recall that if $T$ is in $\operatorname{GL}_n(\mathbb{Z})$ then the form $G(\textbf{X})=F(T(\textbf{X}))$ is said to be equivalent to $F$. Then $V_F=V_G$ and $N^*_G(m)=N^*_F(m)$ for each positive integer $m$.

\begin{lem} \label{lem2}
If $F(\textbf{X})$  is of essentially finite type and $\mathcal{H}(F)$ is minimal among forms equivalent to $F$ then there is a positive number $C_2=C_2(n,d)$, which depends on $n$ and $d$, such that for every $\textbf{x}$ in $\mathbb{R}^n$ there is an $n$-tuple $(\textbf{L}_{i_1}, ... ,\textbf{L}_{i_n})$ in $I'(F)$ and there is a polynomial $G(\textbf{X})$ in $\mathbb{Z}[\textbf{X}]$ of degree $d_0$, with $d_0\geq d-d_F$, which divides $F(\textbf{X})$ in $\mathbb{Z}[\textbf{X}]$ for which
$$
\frac{\prod^{n}_{j=1}|L_{i_j}(\textbf{x})|}{|det(\textbf{L}_{i_1}^{tr}, ... ,\textbf{L}_{i_n}^{tr})|} < C_2\frac{|F(\textbf{x})|^{1/d}|G(\textbf{x})|^{\frac{n-1}{d_0}}}{\mathcal{H}(F)^{1/d}}.
$$
\end{lem}
\begin{proof}
The result holds if $F(\textbf{x})=0$ since then one of the linear forms $L_i(\textbf{x})$ is $0$. Thus we may suppose that $F(\textbf{x})\neq 0$.

The decomposition \eqref{eq1} of $F$ into a product of linear forms is not determined uniquely since if $t_1, ... ,t_d$ are complex numbers with $t_1...t_d = 1$ then we also have
$$
F(\textbf{X}) =t_1L_{1}(\textbf{X})...t_dL_{d}(\textbf{X}).
$$
We may use \eqref{eq22a} to find, for each integer $r$ with $1\leq r \leq k$, linear forms $U_{r,1}(\textbf{X}), ... ,U_{r,d_r}(\textbf{X})$ for which
$$
F_{r}(\textbf{X}) = U_{r,1}(\textbf{X}) ... U_{r,d_r}(\textbf{X})
$$
Thus we may suppose that each linear form $L_i(\textbf{X})$ in \eqref{eq1} is plus or minus $|C_0|^{\frac{1}{d}}U_{p,q}(\textbf{X})$ for some $p$ with $1 \leq p \leq k$ and $q$ with $1 \leq q \leq d_p$.

 For each $\textbf{x}$ in $\mathbb{R}^n$ for which $F(\textbf{x})\neq 0$ put
$$
L_i^{'}(\textbf{X}) = \frac{|F(\textbf{x})|^{1/d}}{|L_i(\textbf{x})|}L_i(\textbf{X})
$$
for $i=1, ... ,d$ and observe that $F(\textbf{X})= L_1^{'}(\textbf{X}) ... L_d^{'}(\textbf{X})$.

Without loss of generality we may suppose that the forms $L_1(\textbf{X}), ... ,L_{r_1}(\textbf{X})$ are from $\mathbb{R}[\textbf{X}], d=r_1+2r_2$ and $L_{r_1+1}(\textbf{X}), ... ,L_{d}(\textbf{X})$ have complex coefficients and are arranged so that $L_i(\textbf{X})= \overline L_{i+r_2}(\textbf{X})$ for $i=r_1+1, ... ,r_1+r_2.$  Let $E^d$ be the set of vectors $(x_1, ... ,x_d)$ with $x_1, ... ,x_{r_1}$ in $\mathbb{R}$ and $x_{r_1+1}, ... ,x_d$ in $\mathbb{C}$ with $x_i=\overline x_{i+r_2}$ for $i=r_1+1, ... ,r_1+r_2$. Then $E^d$ is $d$-dimensional Euclidean space by means of the usual Hermitian inner product on $\mathbb{C}^d$.

Following the proof of Lemma 6 of \cite{Th1} we let $M$ be the $d\times n$ matrix with rows $\textbf{L}^{'}_{1}  , ... , \textbf{L}^{'}_{d}$ and put
$$
M =  \begin{pmatrix} \textbf{m}_1^{tr} & , ... , &\textbf{m}_n^{tr} \end{pmatrix} .
$$
By our ordering of $L_1(\textbf{X}), ... ,L_{d}(\textbf{X})$ we have $\textbf{m}_1, ... ,\textbf{m}_n$ in $E^d$ and so
\begin{equation} \label{eq23}
\left\Vert \wedge^{n}_{j=1} \textbf{m}_j \right\Vert ^2 = \sum_{I'(F)} | det(\textbf{L}^{' tr}_{i_1} , ... , \textbf{L}^{' tr}_{i_n}) |^2 ,
\end{equation}
where $I^{'}(F)$ denotes the set of linearly independent coefficient vectors $(\textbf{L}^{' }_{i_1} , ... , \textbf{L}^{' }_{i_n})$ with $i_1 < ... <i_n$ and where $\wedge$ denotes the wedge product in the Grassman algebra, see Chapter I of \cite{Sch4}.

Let $\lambda_1 \leq ... \leq \lambda_n$ be the successive minima of the $n$-dimensional lattice $\Lambda$ generated by $\textbf{m}_1, ... , \textbf{m}_n$ in $E^d$ with respect to the unit ball. By Minkowski's Theorem on Successive Minima
\begin{equation} \label{eq24}
\lambda_1^2 ... \lambda_n^2 \ll_n det(\Lambda)^2 = \left\Vert \wedge^{n}_{j=1} \textbf{m}_j \right\Vert ^2 .
\end{equation}

Let $\textbf{z}_1, ... ,\textbf{z}_n$ be a basis for $\Lambda$ for which $\left\Vert  \textbf{z}_j \right\Vert \leq j\lambda_j$ for $j=1, ... ,n$, see p.191 of \cite{Sch4}. Then there is a $T$ in $\operatorname{GL}_n(\mathbb{Z})$ with $T=(\textbf{a}_1^{tr}, ... , \textbf{a}_n^{tr})$ for which
$$
MT = (\textbf{z}_1^{tr}, ... , \textbf{z}_n^{tr}).
$$

Put $\textbf{z}_j = (z_{j,1}, ... ,z_{j,d})$ and observe that $z_{j,i} = L_i^{'}(\textbf{a}_j)$ for $j= 1, ... ,n$ and $i= 1, ... ,d$.
 We have
\begin{equation} \label{eq25}
\lambda_1^2 \geq \left\Vert \textbf{z}_1 \right\Vert ^2 = \sum _{i=1}^d |z_{1,i}|^2 = \sum _{i=1}^d |L_i^{'}(\textbf{a}_1)|^2.
\end{equation}
It is at this point that we modify Thunder's proof of Lemma 6 of \cite{Th1}.

Let $G$ be the primitive polynomial in $\mathbb{Z}[\textbf{X}]$ of largest degree which divides $F$ and for which $G(\textbf{a}_1)$ is non-zero. By \eqref{eq10} and the definition of $d_F$, see \eqref{eq11}, the degree $d_0$ of $G$ is at least $d - d_F$. Let $i_1, ... ,i_{d_0}$ be such that
\begin{equation} \label{eq26}
L_{i_1}(\textbf{X}) ... L_{i_{d_0}}(\textbf{X}) = \pm |C_0|^{\frac {d_0}{d}} G(\textbf{X}).
\end{equation}
Then by \eqref{eq25}
$$
\lambda_1^2 \geq  \sum _{j=1}^{d_0} |L_{i_j}^{'}(\textbf{a}_1)|^2
$$
and by the arithmetic-geometric mean inequality
$$
\lambda_1^2 \geq  d_0\prod _{j=1}^{d_0} |L_{i_j}^{'}(\textbf{a}_1)|^{\frac {2}{d_0}}.
$$
Thus, by \eqref{eq26}, 
$$
\lambda_1^2 \geq d_0\frac{|F(\textbf{x})|^{\frac{2}{d}}}{|G(\textbf{x})|^{\frac{2}{d_0}}}|G(\textbf{a}_1)|^{\frac{2}{d_0}}
$$
and since $G(\textbf{a}_1)$ is a non-zero integer we find that
\begin{equation} \label{eq27}
(\lambda_1 ... \lambda_{n-1})^2 \geq \lambda_1^{2(n-1)} \geq \frac{|F(\textbf{x})|^{\frac {2(n-1)}{d}}}{|G(\textbf{x})|^{\frac{2(n-1)}{d_0}}}.
\end{equation}

Now
$$
n^3\lambda_n^2 \geq \sum_{j=1}^{n} (j\lambda_j)^2 \geq \sum_{j=1}^{n} \left\Vert \textbf{z}_j \right\Vert ^2 = \sum_{j=1}^{n} \sum_{i=1}^{d} |z_{j,i}|^2
$$
so
$$
n^3\lambda_n^2 \geq \sum_{i=1}^{d} \left\Vert (z_{1,i}, ... ,z_{n,i})\right\Vert ^2
$$
and by the arithmetic-geometric mean inequality
$$
n^3\lambda_n^2 \geq d(\prod_{i=1}^{d} \left\Vert (z_{1,i}, ... ,z_{n,i})\right\Vert ^2)^{\frac{1}{d}} \geq d(\mathcal{H}(F \circ T))^{\frac{2}{d}}.
$$
Since we have assumed that $\mathcal{H}(F)$ is minimal among forms equivalent to it we have
\begin{equation} \label{eq28}
\lambda_n^2 \geq \frac{d}{n^3}(\mathcal{H}(F))^{\frac{2}{d}}.
\end{equation}

By \eqref{eq23}, \eqref{eq24}, \eqref{eq27} and  \eqref{eq28}
$$
\sum_{I'(F)} | det(\textbf{L}^{' tr}_{i_1} , ... , \textbf{L}^{' tr}_{i_n} )|^2  \gg_{n,d} \frac{|F(\textbf{x})|^{\frac {2(n-1)}{d}}}{|G(\textbf{x})|^{\frac{2(n-1)}{d_0}}}\mathcal{H}(F)^{\frac{2}{d}}.
$$
The number of terms in the sum above is at most $\binom {d}{n}$ hence there is an $n$-tuple $(i_1, ... ,i_n)$ for which
$$
| det(\textbf{L}^{' tr}_{i_1} , ... , \textbf{L}^{' tr}_{i_n}) |  \gg_{n,d} \frac{|F(\textbf{x})|^{\frac {n-1}{d}}}{|G(\textbf{x})|^{\frac{n-1}{d_0}}}\mathcal{H}(F)^{\frac{1}{d}}.
$$
Since
$$
| det(\textbf{L}^{' tr}_{i_1} , ... , \textbf{L}^{' tr}_{i_n}) | = \frac{| det(\textbf{L}^{ tr}_{i_1} , ... , \textbf{L}^{ tr}_{i_n}) | |F(\textbf{x})|^{\frac{n}{d}}}{\prod_{j=1}^{n}|L_{i_j}(\textbf{x})|}
$$
the result follows.

\end{proof}

\section{Bounds from the Subspace Theorem}

Let $\mathbb{K}$ be an algebraic number field of degree $d$ over $\mathbb{Q}$ with ring of algebraic integers  $\mathcal{O}_{\mathbb{K}}$ and let $M(\mathbb{K})$ be the set of equivalence classes of absolute values on $\mathbb{K}$. Suppose that $\mathbb{K}$ has $r_1$ real embeddings $\sigma_1, ... ,\sigma_{r_1}$ and $2r_2$ pairs of complex embeddings with $\sigma_{r_1+i} = \overline {\sigma_{r_1 +r_2 + i}}$ for $i = 1, ... ,r_2.$ Then $M(\mathbb{K})$ consists of $r_1+r_2$ Archimedean absolute values represented by, for each $x$ in $\mathbb{K}$,
$$
|x|_{v_i} = |\sigma_i(x)|
$$
for $i= 1, ... ,r_1 + r_2$. We put $n_{v_i} = 1$ for $i=1, ... ,r_1$ and $n_{v_i}=2$ for $i=r_1+1, ... ,r_1 + r_2.$ The non-Archimedean absolute values correspond to prime ideals $\mathcal{P}$ of $\mathcal{O}_{\mathbb{K}}$. For each prime ideal $\mathcal{P}$ and each $x$ in $\mathbb{K}$ with $x\neq 0$
$$
|x|_{v_{\mathcal{P}}} = p^{-(ord_{\mathcal{P}} x)/e}
$$
where $ord_{\mathcal{P}} x$ denotes the order of $\mathcal{P}$ in the fractional ideal generated by $x$ and where $e$ is the index of ramification. We put $n_{v_\mathcal{P}} = ef$ where $f$ is the residue field degree.

We extend this to $n$-tuples $\textbf{x} = (x_1, ... ,x_n)$ with $\textbf{x} \neq \textbf{0}$ and $x_i$ in $\mathbb{K}$ for $i=1, ... ,n$ by 
$$
|\textbf{x}|_{v} = (\sum_{i=1}^{n} |x_i|_{v}^{2})^{1/2}
$$
if $v$ is infinite,
$$
|\textbf{x}|_{v} = max( |x_1|_v, ... ,|x_n|_v)
$$
if $v$ is finite.

We now define the field height $H_{\mathbb{K}}(\textbf{x})$ of $\textbf{x}$ by
$$
H_{\mathbb{K}}(\textbf{x}) = \prod_{v}|\textbf{x}|_v^{n_v}.
$$
Let $L(\textbf{X}) = \alpha_1X_1 + ... + \alpha_nX_n$ be a linear form with coefficients in $\mathbb{K}$. We define the field height of $L(\textbf{X})$, $H_{\mathbb{K}}(L) $ by
$$
H_{\mathbb{K}}(L) = H_{\mathbb{K}}((\alpha_1, ... ,\alpha_n))
$$
and the absolute height $H(L)$ by
$$
H(L) = H_{\mathbb{K}}(L)^{\frac{1}{d}}
$$
where $d$ is the degree of ${\mathbb{K}}$ over $\mathbb{Q}$.

Suppose that $R$ is an irreducible decomposable form of content 1 as in \eqref{eq22}. By Lemma 2a, Chapter III, of \cite{Sch4} 
$$
H_{\mathbb{K}}(AL) = \frac {1}{m} \prod_{\sigma} \left\Vert \sigma (AL)\right\Vert .
$$
Thus
$$
\mathcal{H} (R) = H_{\mathbb{K}}(AL) = H(AL)^d
$$
By the product formula $H(AL)= H(L)$ and so
\begin{equation} \label{eq29}
\mathcal {H}(R)  \geq H(L).
\end{equation}

Suppose that $F$ has the form \eqref{eq1} and \eqref{eq10} and that $L_1(\textbf{X}), ... ,L_d(\textbf{X})$ are as in the proof of Lemma 6. Then, by the product formula and \eqref{eq29}
\begin{equation} \label{eq30}
\mathcal {H}(F)  \geq H(L_{i})
\end{equation}
for $i=1, ... ,d$ and, as a consequence,
\begin{equation} \label{eq30a}
\mathcal {H}(F)  \geq 1.
\end{equation}
Also note that if $T$ is in $\operatorname{GL}_n(\mathbb{Z})$ then $F \circ T =(L_1 \circ T) ... (L_n \circ T) $ and \eqref{eq30} holds with $F$ replaced by $F \circ T$ and $L_i$ by $L_i \circ T$ for $i=1, ... ,d.$

In 1989 Schmidt \cite{Sch2} established a quantitative version of the Subspace Theorem \cite{Sch1}. This was subsequently refined by Evertse \cite{JH2} who proved the following.

\begin{lem} \label{lem3}
Let $L_1, ... ,L_n$ be linearly independent linear forms in $n$ variables with algebraic coefficients, $H(L_i) \leq H$ and $[\mathbb{Q}(L_i):\mathbb{Q}] \leq D$ for $i= 1, ... ,n$. For every $\delta$ with $0 < \delta <1$ there are $t$ proper rational subspaces $T_1, ... ,T_t$ of $\mathbb{Q}^n$ with
\begin{equation} \label{eq31a}
t \leq 2^{60n^2}\delta^{-7n} \log 4D \log \log 4D
\end{equation}
such that every primitive integral solution $\textbf{x}$ of
\begin{equation} \label{eq31}
\frac{|L_{1}(\textbf{x}) ... L_{n}(\textbf{x})|}{|det(\textbf{L}_{1}^{tr}, ... ,\textbf{L}_{n}^{tr})|} < \left\Vert \textbf x \right\Vert ^{-\delta}
\end{equation}
with $H(\textbf{x}) \geq H$ lies in $T_1 \cup ... \cup T_t.$

\end{lem}
\begin{proof}
This is the Corollary of the main theorem of Evertse \cite{JH2}.

\end{proof}

We shall use Lemma \ref{lem1} and Lemma \ref{lem3} to prove our next result.

\begin{lem} \label{lem4}
Let $F$ be a decomposable form in $n$ variables with integer coefficients and degree $d$ with $d > n \geq 2$ as in \eqref{eq1}. Suppose that $F$ is of essentially finite type and that $F$ is not proportional to a power of a definite quadratic form in $2$ variables. Put
\begin{equation} \label{eq32}
C = max(C_1, m^{\frac{1}{a_F}}, m^{\frac{1}{d}}\mathcal{H}_0(F)^{1+c_F})^{4a_F(n-1)}
\end{equation}
where $C_1$ is given in Lemma \ref{lem1}. There is a positive number $c$, which is computable in terms of $n$ and $d$, and there are $t$ proper rational subspaces $T_1, ... ,T_t$ of $\mathbb{Q}^n$ with $t \leq c$ such that if $\textbf{a}$ is an integer point with $\left\Vert \textbf a \right\Vert \geq C$ for which
$$
1 \leq |F(\textbf{a})| \leq m
$$
then $\textbf{a}$ is in $T_1\cup ... \cup T_t$.

\end{lem}
\begin{proof}
Since $F$ is of essentially finite type and $F$ is not proportional to a power of a definite quadratic form in $2$ variables by \eqref{eq15}
\begin{equation} \label{eq33}
\frac{d-na_F}{a_F} \geq \frac{1}{a_F(n-1)} \geq \frac {n}{d(n-1)}.
\end{equation}
Put
$$
\delta = min(\frac{1}{2} , \frac{d-na_F}{4a_F})
$$
so that
$$
\frac{n}{4d(n-1)} \leq \delta \leq \frac{1}{2}.
$$

We may suppose, without loss of generality, that $\mathcal{H}(F)=\mathcal{H}_0(F).$ If $\textbf{a}$ is an integer point with $\left\Vert \textbf a \right\Vert > C$ for which $1 \leq |F(\textbf{a})| \leq m$ then by Lemma \ref{lem1} there is an $n$-tuple $(\textbf{L}_{i_1}, ... ,\textbf{L}_{i_n})$ in $J(F)$ such that
$$
\frac{\prod^{n}_{j=1}|L_{i_j}(\textbf{a})|}{|det(\textbf{L}_{i_1}^{tr}, ... ,\textbf{L}_{i_n}^{tr})|} < C_1\bigg (\frac{m}{\left\Vert \textbf a \right\Vert ^{d-na_F}}\bigg )^{1/a_F}\mathcal{H}(F)^{c_F}.
$$
By \eqref{eq33}
$$
\left\Vert \textbf a \right\Vert ^{\frac{d-na_F}{4a_F}} \geq \left\Vert \textbf a \right\Vert ^{\frac{1}{4a_F(n-1)}} \geq max(C_1, m^{\frac{1}{a_F}}, m^{\frac{1}{d}}\mathcal{H}_0(F)^{1+c_F}).
$$
Thus
$$
\frac{\prod^{n}_{j=1}|L_{i_j}(\textbf{a})|}{|det(\textbf{L}_{i_1}^{tr}, ... ,\textbf{L}_{i_n}^{tr})|} < \frac{1}{\left\Vert \textbf a \right\Vert ^{\frac{d-na_F}{4a_F}}} \leq \frac {1}{\left\Vert \textbf a \right\Vert ^{\delta}}.
$$

We may write $\textbf{a} =g\textbf{a}^{'}$ with $\textbf{a}^{'}$ primitive. Since $1 \leq |F(\textbf{a})| \leq m$ and $|F(\textbf{a})| = g^d|F(\textbf{a}^{'})|$,
\begin{equation} \label{eq34}
g \leq m^{\frac{1}{d}}.
\end{equation}

Therefore
$$
\frac{\prod^{n}_{j=1}|L_{i_j}(\textbf{a}^{'})|}{|det(\textbf{L}_{i_1}^{tr}, ... ,\textbf{L}_{i_n}^{tr})|} <   \frac {1}{\left\Vert \textbf a^{'} \right\Vert ^{\delta}}.
$$
But $\left\Vert \textbf a \right\Vert > C$ hence by \eqref{eq17} and \eqref{eq34}
$$
\left\Vert \textbf a^{'} \right\Vert \geq \mathcal{H}(F)^{1+c_F} \geq \mathcal{H}(F).
$$
Further, by \eqref{eq30}, $\left\Vert \textbf a^{'} \right\Vert \geq H(L_{i_j})$ for $j = 1, ... ,n.$ Thus by Lemma \ref{lem3} $\left\Vert \textbf a^{'} \right\Vert $, hence also $\left\Vert \textbf a \right\Vert$, is in one of at most $|J(F)|c_1$ proper subspaces of $\mathbb{Q}^n$, where $c_1$ is the right hand side of inequality \eqref{eq31a} with $d$ in place of $D$.

\end{proof}

\section{Bounds for integer points}

We require an estimate for the number of integer points $\textbf{a}$ for which an expression of the form
$$
\frac{\prod^{n}_{j=1}|L_{i_j}(\textbf{a})|}{|det(\textbf{L}_{i_1}^{tr}, ... ,\textbf{L}_{i_n}^{tr})|} 
$$
is small.

\begin{lem} \label{lem5}
Let $K_1(\textbf{X}), ... ,K_n(\textbf{X})$ be $n$ linearly independent linear forms in $\mathbb{C}[\textbf{X}]$. Let $A$ and $C$ be positive real numbers and let $U$ be the set of $\textbf{x}$ in $\mathbb{R}^n$ with
$$
\frac{\prod^{n}_{i=1}|K_{i}(\textbf{x})|}{|det(\textbf{K}_{1}^{tr}, ... ,\textbf{K}_{n}^{tr})|} < A
$$
and $\left\Vert \textbf x \right\Vert \leq C$. Let $D$ be a real number with $D > 1$. The set of integer points in $U$ lie in a set $W$ or in at most $\kappa$ proper subspaces where if $C^n \geq n!AD^n$ then
$$
 |W|\ll_n \bigg(\frac{\log (\frac{C^n}{n!A})}{\log D}\bigg)^{n-1}AD^{n+1}
$$
and
$$
\kappa \ll_n  \bigg(\frac{\log (\frac{C^n}{n!A})}{\log D}\bigg)^{n-1}
$$
and where otherwise
$$
|W| \ll_n AD^{n+1}
$$
and 
$$
\kappa  \ll_n 1.
$$

\end{lem}
\begin{proof}
This follows from Lemma 7 of \cite{Th2} and Lemma 9 of \cite{Th1}.
\end{proof}

\begin{lem} \label{lem6}
Let $K_1(\textbf{X}), ... ,K_n(\textbf{X})$ be $n$ linearly independent linear forms in $\mathbb{C}[\textbf{X}]$. Let $A$ and $B$ be positive real numbers and let $U$ be the set of $\textbf{x}$ in $\mathbb{R}^n$ with
$$
\frac{\prod^{n}_{i=1}|K_{i}(\textbf{x})|}{|det(\textbf{K}_{1}^{tr}, ... ,\textbf{K}_{n}^{tr})|} < A
$$
and
$$
B \leq \left\Vert \textbf x \right\Vert \leq 2B.
$$
 The set of integer points in $U$ lie in a set $W$ or in at most $\kappa$ proper subspaces where if $B^n \geq n^{\frac{n}{2}}n!A$ then
$$
 |W|\ll_n \bigg(\log (\frac{2^nB^n}{n^{\frac{n}{2}}n!A})\bigg)^{n-2}A
 $$
and
$$
\kappa \ll_n   \bigg(\log (\frac{2^nB^n}{n^{\frac{n}{2}}n!A})\bigg)^{n-2}
$$
and where otherwise
$$
|W| \ll_n A
$$
and 
$$
\kappa  \ll_n 1.
$$

\end{lem}
\begin{proof}
This follows from Lemma 6 of \cite{Th2} with $C=2B$ and $D=2$  and Lemma 9 of \cite{Th1}.
\end{proof}

\section{Powers of a definite quadratic forms in $2$ variables}

Suppose $F$ is proportional to a power of a definite quadratic form in $2$ variables, say
\begin{equation} \label{eq35}
F(X_1,X_2) = h(AX_1^2 + BX_1X_2 + CX_2^2)^k
\end{equation}
with $h,k,A,B,C$ integers with $h\neq 0, k\geq 2$ and $B^2-4AC < 0.$ Then $F$ is of essentially finite type but $a_F=k=d/2$ where $d$ denotes the degree of $F$ and so \eqref{eq15} does not hold. As a consequence we must estimate $N^*_F(m)$ directly.

We proceed by first estimating $N^*_G(m)$ where $G(X_1,X_2) = AX_1^2 + BX_1X_2 + CX_2^2$. In 1915 Landau \cite{L} gave an asymptotic estimate for $N_G(m)$, hence also for $N^*_G(m)$, however he did not make explicit the dependence on the coefficients of $G$ in his estimate, a feature that we require.

\begin{lem} \label{lem7}
Let $F$ be a form as in \eqref{eq35}. Then
\begin{equation} \label{eq36}
|N^*_F(m)-\frac{2\pi}{\sqrt{4AC-B^2}}(\frac{m}{h})^{2/d}| \ll(\frac{m}{h})^{1/d}.
\end{equation}

\end{lem}

\begin{proof}
We first remark that
\begin{equation} \label{eq37}
N^*_F(m)=N^*_G((\frac{m}{h})^{1/k})
\end{equation}
and, since $G$ is definite, that for any positive real number $q$
\begin{equation} \label{eq38}
N_G(q) = 1 + N^*_G(q).
\end{equation}
We may suppose that $G$ is positive definite. The argument when $G$ is negative definite is similar.

We note that
\begin{equation} \label{eq39}
N_{G}(q)=N_{\tilde G}(q)
\end{equation}
whenever $\tilde G(X_1,X_2) = G(aX_1+bX_2, cX_1+dX_2)$ where $(\begin{smallmatrix}a & b \\ c & d \end{smallmatrix})$ is in $\operatorname{SL}_2(\mathbb{Z})$. By choosing $(\begin{smallmatrix}a & b \\ c & d \end{smallmatrix})$ appropriately we may suppose that $\tilde G$ is in reduced form. In particular
$$
\tilde G(X_1,X_2) = rX_1^2 + sX_1X_2 + tX_2^2
$$
with $|s| \leq r \leq t.$ We also have
\begin{equation} \label{eq40}
s^2-4rt = B^2-4AC
\end{equation}
since $(\begin{smallmatrix}a & b \\ c & d \end{smallmatrix})$ is in $\operatorname{SL}_2(\mathbb{Z})$.

The region $\{(x_1,x_2) \in \mathbb R^2 : |\tilde G(x_1,x_2)| \leq q \}$ is closed and bounded with perimeter of length $P$. Thus, see \cite{Daven},
\begin{equation} \label{eq41}
|N_{\tilde G}(q) - V_{\tilde G}q| < 1+ 4(P+1),
\end{equation}
and, by \eqref{eq40},
\begin{equation} \label{eq42}
V_{\tilde G} = \frac{2\pi}{\sqrt {4rt-s^2}} = \frac{2\pi}{\sqrt {4AC-B^2}} .
\end{equation}
The length of the perimeter of the ellipse given by $\tilde G = q$ is at most $2\pi w$ where 
$$
w=\big ( \frac{2(r+t+\sqrt{(r-t)^2 +s^2})}{4rt-s^2}\big )^{1/2}q^{1/2}.
$$
Since $|s| \leq r \leq t$ we have $4rt-s^2 \geq 3rt$ and so $w \ll q^{1/2}.$ Thus by \eqref{eq41} and \eqref{eq42}
$$
|N_{\tilde G}(q) -  \frac{2\pi}{\sqrt {4AC-B^2}}q| \ll q^{1/2}
$$
and so by \eqref{eq38} and \eqref{eq39}
$$
|N_{G}^{*}(q) -  \frac{2\pi}{\sqrt {4AC-B^2}}q| \ll q^{1/2}
$$
hence, since $k=d/2$, the result follows from \eqref{eq37}.

\end{proof}

\section{Small solutions}

Let $m$ and $B$ be real numbers with $m \geq 1$ and $B \geq 1$. Denote by $Vol(m,B)$ the volume of the set
$$
\{(x_1,...,x_n) \in \mathbb R^n : |F(x_1, ... ,x_n)| \leq m \text {  and  } \left\Vert \textbf x \right\Vert \leq B \}.
$$

\begin{lem} \label{lem8}
Let $m$ and $B$ be real numbers with $m \geq 1$ and $B \geq 1$. If $F$, as in \eqref{eq1}, is of essentially finite type and $F$ is not proportional to a power of a definite quadratic form in $2$ variables then
$$
|V_Fm^{n/d} - Vol(m,B)| \ll_{n,d} \mathcal{H}(F)^{c_F}m^{1/a_F}B^{(na_F-d)/a_F}(1 + \log B)^{n-2}.
$$
\end{lem}

\begin{proof}
This follows from \eqref{eq15} and the proof of Lemma 15 of \cite{Th1}.

\end{proof}

Let $Vol^{'}(m,B)$ be the volume of the set
$$
\{(x_1,...,x_n) \in \mathbb R^n : |F(x_1, ... ,x_n)| \leq m \text {  and  } max(|x_1|, ... ,|x_n|)  \leq B \}
$$
and let $N_F^{'}(m,B)$ be the number of integer points $\textbf{a}$ for which $|F(\textbf{a}| \leq m$ and $ max(|x_1|, ... ,|x_n|)  \leq B$.

\begin{lem} \label{lem9}
Let $m$ and $B$ be real numbers with $m \geq 1$ and $B \geq 1$. If $F$, as in \eqref{eq1}, is of essentially finite type then
$$
|N^{'}_F(m,B)  - Vol^{'}(m,B)| \leq dn(2B+1)^{n-1}.$$
\end{lem}

\begin{proof}
This follows from  the proof of Lemma 14 of \cite{Th1}.

\end{proof}

\section{Large solutions}

In order to establish Theorems 1 and 2 we shall first analyze the set $Q(m)$ of integer points $\textbf{a}$ for which
$$
1 \leq |F(a_1, ... ,a_n)| \leq m
$$
and
$$
\left\Vert \textbf a \right\Vert \geq m^{1/(d-a_F)}.
$$

\begin{lem} \label{lem10}
Let F be a non-zero decomposable form in $n$ variables with integer coefficients and degree $d$ with $d > n \geq 2$ and suppose that $F$ is of essentially finite type and not proportional to a power of a definite quadratic form in $2$ variables. Let $m$ be a positive integer. $Q(m)$ is contained in a set of cardinality 
$$
\ll_{n,d} \mathcal{H}(F)^{c_F}(1 + \log m)^{n-2}m^{\frac{n-1}{d-a_F}}
$$
or in
$$
\ll_{n,d} (1+\log m + \log \mathcal{H}(F))^{n-1}
$$
proper subspaces.

\end{lem}

\begin{proof}

By Lemma \ref{lem4} if $\textbf{a}$ is in $Q(m)$ with $\left\Vert \textbf a \right\Vert \geq C$ where $C$ is given by \eqref{eq32} then $\textbf{a}$ lies in one of
\begin{equation} \label{eq43a}
\ll_{n,d} 1
\end{equation}
proper rational subspaces.

It remains to consider the integer points $\textbf{a}$ in $Q(m)$ with
$$
m^{1/(d-a_F)} \leq \left\Vert \textbf a \right\Vert \leq C.
$$
and we do so by controling the solutions with
\begin{equation} \label{eq43}
2^{j-1}m^{1/(d-a_F)} \leq \left\Vert \textbf a \right\Vert \leq 2^{j}m^{1/(d-a_F)}
\end{equation}
for $j=1, ... ,t$ where $t$ satisfies
$$
2^{t-1} < C \leq 2^t.
$$
By \eqref{eq32},
\begin{equation} \label{eq44}
t \ll_{n,d} 1 + \log m + \log \mathcal{H}_0(F) .
\end{equation}

By Lemma \ref{lem1} and \eqref{eq15} , for each integer point $\textbf{a}$ for which \eqref{eq43} holds and $1 \leq |F(\textbf{a})| \leq m$ there is an $n$-tuple $(\textbf{L}_{i_1}, ... , \textbf{L}_{i_n})$ in $J(F)$ for which
$$
\frac{\prod^{n}_{j=1}|L_{i_j}(\textbf{a})|}{|det(\textbf{L}_{i_1}^{tr}, ... ,\textbf{L}_{i_n}^{tr})|} < \frac {C_1m^{\frac {n-1}{d-a_F}}\mathcal{H}(F)^{c_F}}{2^{\frac{j-1}{a_F(n-1)}}}.
$$
We may now apply Lemma \ref{lem6} with $A=\frac{C_1m^{\frac {n-1}{d-a_F}}\mathcal{H}(F)^{c_F}}{2^{\frac{j-1}{a_F(n-1)}}}$ and $B= 2^{j-1}m^{\frac {1}{d-a_F}}$. We find that such integer points $\textbf{a}$ lie in a set of cardinality 
\begin{equation} \label{eq45}
\ll_{n,d} (1 + \log m + j)^{n-2}\frac {m^{\frac {n-1}{d-a_F}}\mathcal{H}(F)^{c_F}}{2^{\frac{j-1}{a_F(n-1)}}}.
\end{equation}
or, since $j \leq t$ and \eqref{eq44} holds, in 
\begin{equation} \label{eq46}
\ll_{n,d} (1 + \log m + \log \mathcal{H}(F))^{n-2}
\end{equation}
proper subspaces.

Observe that, for $j \geq 3$,
\begin{equation} \label{eq47}
(1 + \log m + j) \leq (1+ \log m)j
\end{equation}
and that
\begin{equation} \label{eq48}
\sum_{j=1}^{t} \frac {j^{n-2}}{2^{\frac{j-1}{a_F(n-1)}}} \ll_{n,d}  1.
\end{equation}
Since $|J(F)| \leq n!\binom{d}{n}$ it follows from \eqref{eq43a}, \eqref{eq45}, \eqref{eq47}  and \eqref{eq48}, that the set of integer points $\textbf{a}$ with $m^{1/(d-a_F)} \leq \left\Vert \textbf a \right\Vert$ and $1 \leq |F(\textbf{a})| \leq m$ lies in a set of cardinality
$$
\ll_{n,d} (1+ \log m)^{n-2}m^{\frac {n-1}{d-a_F}}\mathcal{H}(F)^{c_F}
$$
or in
$$
\ll_{n,d} (1+ \log m + \log \mathcal{H}(F))^{n-1}
$$
proper subspaces as required.

\end{proof}

\section{Proof of Theorem 1}

If $F$ is proportional to the power of a definite quadratic form in $2$ variables then the result holds by Lemma \ref{lem7} since $|h| \geq 1, |4AC-B^2| \geq 1$ and $d_F=0$ and so we may assume this is not the case.

Since $N^*_G(m) = N^*_F(m)$ when $G$ is equivalent to $F$ we may assume, without loss of generality, that $\mathcal{H}(F) = \mathcal{H}_0(F)$.

We shall first prove that the integer points $\textbf{a}$ with $1 \leq |F(\textbf{a})| \leq m$ lie in a set of cardinality
\begin{equation} \label{eq49}
\ll_{n,d} m^{\frac{1}{d} + \frac{n-1}{d-d_F}}.
\end{equation}
or in at most
\begin{equation} \label{eq50}
\ll_{n,d} (1 + \log m)^{n-2}
\end{equation}
proper subspaces of $\mathbb{R}^n$.

To do so we consider two cases. In the first case
\begin{equation} \label{eq51}
\mathcal{H}(F)^{c_F} \geq m^{\frac{n}{2(d^2(n-1)^2 + d)}}/(1 + \log m)^{n-2},
\end{equation}
and in the second case
\begin{equation} \label{eq52}
\mathcal{H}(F)^{c_F} < m^{\frac{n}{2(d^2(n-1)^2 + d)}}/(1 + \log m)^{n-2}.
\end{equation}

Suppose initially that \eqref{eq51} holds. Let $\textbf{a}$ be an integer point with $1 \leq |F(\textbf{a})| \leq m$. By Lemma \ref{lem2} there is an $n$-tuple $(\textbf{L}_{i_1}, ... , \textbf{L}_{i_n})$ in $I(F)$ and a polynomial $G$ in $\mathbb{Z}[\textbf{X}]$ of degree $d_0 \geq d-d_F$ which divides $F$  for which
$$
\frac{\prod^{n}_{j=1}|L_{i_j}(\textbf{a})|}{|det(\textbf{L}_{i_1}^{tr}, ... ,\textbf{L}_{i_n}^{tr})|} < C_2\frac{|F(\textbf{a})|^{1/d}|G(\textbf{a})|^{\frac{n-1}{d_0}}}{\mathcal{H}(F)^{1/d}}.
$$
Since $G(\textbf{a})$ divides $F(\textbf{a})$ and $F(\textbf{a})$ is a non-zero integer of size at most $m$ we see that $|G(\textbf{a})| \leq m$ and so
\begin{equation} \label{eq53}
\frac{\prod^{n}_{j=1}|L_{i_j}(\textbf{a})|}{|det(\textbf{L}_{i_1}^{tr}, ... ,\textbf{L}_{i_n}^{tr})|} < C_2\frac{m^{\frac{1}{d} + \frac{n-1}{d_0}}}{\mathcal{H}(F)^{1/d}}.
\end{equation}

Thus each integer point $\textbf{a}$ with $1 \leq |F(\textbf{a})| \leq m$ lies in a set of $\textbf{x}$ in $\mathbb{R}^n$ with
$$
\frac{\prod^{n}_{j=1}|L_{i_j}(\textbf{x})|}{|det(\textbf{L}_{i_1}^{tr}, ... ,\textbf{L}_{i_n}^{tr})|} < C_2\frac{m^{\frac{1}{d} + \frac{n-1}{d_0}}}{\mathcal{H}(F)^{1/d}}.
$$
for some $n$-tuple $(\textbf{L}_{i_1}, ... , \textbf{L}_{i_n})$ from $I(F)$. We now apply Lemma \ref{lem5} with $A= C_2\frac{m^{\frac{1}{d} + \frac{n-1}{d_0}}}{\mathcal{H}(F)^{1/d}}, C$ given by \eqref{eq32} and $D= (2\mathcal{H}(F))^{\frac{1}{(n+1)d}}$ to conclude that the set of integer points $\textbf{a}$ with $\left\Vert \textbf a \right\Vert \leq C$ for which \eqref{eq53} holds lie in a set of cardinality
$$
\ll_{n,d} m^{\frac{1}{d} + \frac{n-1}{d_0}}
$$
or in
$$
\ll_{n,d} 1
$$
proper subspaces.

By Lemma \ref{lem4} the set of integer points $\textbf{a}$ with $1 \leq |F(\textbf{a})| \leq m$ for which $\left\Vert \textbf a \right\Vert > C$ lie in 
$$
\ll_{n,d} 1
$$
proper rational subspaces. Therefore, provided that \eqref{eq51} holds, the integer points $\textbf{a}$ with $1 \leq |F(\textbf{a})| \leq m$ either lie in a set of cardinality
\begin{equation} \label{eq55}
\ll_{n,d} m^{\frac{1}{d} +\frac {n-1}{d-d_F}}
\end{equation}
or in
\begin{equation} \label{eq56}
\ll_{n,d} 1
\end{equation}
proper subspaces.

We now suppose that we are in the second case, so that \eqref{eq52} holds. For any real number $x$ let $\lfloor x \rfloor$ denote the greatest integer less than or equal to $x$. Put
$$
r = \lfloor \mathcal{H}(F)^{2c_F(d^2(n-1)^2 + d)/n} \rfloor .
$$
For $x \geq 1$ we have $\lfloor x \rfloor \geq x/2$ hence
\begin{equation} \label{eq57}
r \geq  \mathcal{H}(F)^{2c_F(d^2(n-1)^2 + d)/n} /2 .
\end{equation}
since $\mathcal{H}(F) \geq 1$ by \eqref{eq30a}. Observe that
$$
r^{\frac{n}{2(d^2(n-1)^2 + d)}} \leq \mathcal{H}(F)^{c_F}
$$
and so \eqref{eq51} holds with $r$ in place of $m$. Thus, by \eqref{eq55} and \eqref{eq56}, the set of integer points $\textbf{a}$ with $1 \leq |F(\textbf{a})| \leq r$ either lie in a set of cardinality
$$
\ll_{n,d} r^{\frac{1}{d} +\frac {n-1}{d-d_F}}
$$
or in 
$$
\ll_{n,d} 1
$$
proper subspaces. Each proper subspace contains at most $(1+2r^{\frac {1}{d-a_F}})^{(n-1)}$ integer points $\textbf{a}$ for which $\left\Vert \textbf a \right\Vert \leq r^{\frac{1}{d-a_F}}$. By \eqref{eq15} $ \frac{n-1}{d-a_F} < \frac {n}{d}$ and so the number of integer points $\textbf{a}$ with $1 \leq |F(\textbf{a})| \leq r$ and $\left\Vert \textbf a \right\Vert \leq r^{\frac{1}{d-a_F}}$ is 
$$
\ll_{n,d} r^{\frac{1}{d} + \frac{n-1}{d-d_F}}
$$
hence the number $N_F(r, r^{\frac{1}{d-a_F}})$ of integer points $\textbf{a}$ with $|F(\textbf{a})| \leq r$ and $\left\Vert \textbf a \right\Vert \leq r^{\frac{1}{d-a_F}}$ is
$$
\ll_{n,d} r^{\frac{1}{d} +\frac {n-1}{d-d_F}}.
$$
Therefore the number $N^{'}_F(r, r^{\frac{1}{d-a_F}})$ of integer points $\textbf{a}$ with $|F(\textbf{a})| \leq r$ and
$$
max (|a_1|, ... ,|a_n|) \leq r^{\frac{1}{d-a_F}}
$$
is
$$
\ll_{n,d} r^{\frac{1}{d} +\frac {n-1}{d-d_F}}.
$$

By Lemma \ref{lem9} with $m=r$ and $B= r^{\frac{1}{d-a_F}}$ we find that
$$
|Vol^{'}(r, r^{\frac{1}{d-a_F}}) - N^{'}_F(r, r^{\frac{1}{d-a_F}})| \ll_{n,d} r^{\frac{n-1}{d-a_F}}
$$
hence
$$
Vol^{'}(r, r^{\frac{1}{d-a_F}}) \ll_{n,d} r^{\frac{1}{d} + \frac{n-1}{d-d_F}}.
$$
and so
\begin{equation} \label{eq58}
Vol(r, r^{\frac{1}{d-a_F}}) \ll_{n,d} r^{\frac{1}{d} + \frac{n-1}{d-d_F}}.
\end{equation}

By Lemma \ref{lem8}
$$
| V_Fr^{\frac{n}{d}} - Vol(r, r^{\frac{1}{d-a_F}})| \ll_{n,d} \mathcal{H}(F)^{c_F}r^{\frac{n-1}{d-a_F}}(1 + \log r)^{n-2}.
$$
By \eqref{eq15} 
\begin{equation} \label{eq59}
\frac{n}{d} - \frac{n-1}{d-a_F} \geq \frac{n}{d} - \frac{n-1}{d-\frac {d}{n} + \frac {1}{n(n-1)}} = \frac{n}{d^2(n-1)^2 +d}
\end{equation}
and so
$$
| V_Fr^{\frac{n}{d}} - Vol(r, r^{\frac{1}{d-a_F}})| \ll_{n,d} \frac {r^{\frac{n}{d}}\mathcal{H}(F)^{c_F}(1+ \log r)^{n-2}}{r^{\frac {n}{d^2(n-1)^2 + d}}}
$$
hence by \eqref{eq57}
\begin{equation} \label{eq60}
| V_Fr^{\frac{n}{d}} - Vol(r, r^{\frac{1}{d-a_F}})| \ll_{n,d} r^{\frac{n}{d}},
\end{equation}
and so by \eqref{eq58}
\begin{equation} \label{eq61}
V_Fr^{\frac{n}{d}}  \ll_{n,d} r^{\frac{1}{d} + \frac{n-1}{d-d_F}}.
\end{equation}

Since \eqref{eq52} holds, $m=rs$ with $s > 1$ and so
$$
V_Fm^{\frac{n}{d}} =V_Fr^{\frac{n}{d}}s^{\frac{n}{d}}.
$$
Thus by \eqref{eq61}
$$
V_Fm^{\frac{n}{d}} \ll_{n,d} r^{\frac{1}{d} + \frac{n-1}{d-d_F}}s^{\frac{n}{d}}
$$
hence
\begin{equation} \label{eq62}
V_Fm^{\frac{n}{d}}  \ll_{n,d} m^{\frac{1}{d} + \frac{n-1}{d-d_F}}.
\end{equation}

By Lemma \ref{lem8} with $B=m^{\frac{1}{d-a_F}}$
$$
| V_Fm^{\frac{n}{d}} - Vol(m, m^{\frac{1}{d-a_F}})| \ll_{n,d} \mathcal{H}(F)^{c_F}m^{\frac{n-1}{d-a_F}}(1 + \log m)^{n-2}
$$
and since $m > r$ we find, as in \eqref{eq60}, that
$$
| V_Fm^{\frac{n}{d}} - Vol(m, m^{\frac{1}{d-a_F}})| \ll_{n,d} m^{\frac {n}{d}}
$$
and so by \eqref{eq62}
$$
Vol(m, m^{\frac{1}{d-a_F}}) \ll_{n,d} m^{\frac{1}{d} + \frac{n-1}{d-d_F}}
$$
and
\begin{equation} \label{eq63}
Vol^{'}(m, m^{\frac{1}{d-a_F}}) \ll_{n,d} m^{\frac{1}{d} + \frac{n-1}{d-d_F}}.
\end{equation}
Since $\frac{n-1}{d-a_F} < \frac {n}{d}$,
by Lemma \ref{lem9} with $B=m^{\frac{1}{d-a_F}}$ we find that
$$
|Vol^{'}(m, m^{\frac{1}{d-a_F}}) - N^{'}_F(m, m^{\frac{1}{d-a_F}})| \ll_{n,d} m^{\frac {n}{d}}.
$$
Thus, by \eqref{eq63},
$$
N^{'}_F(m, m^{\frac{1}{d-a_F}}) \ll_{n,d} m^{\frac{1}{d} + \frac{n-1}{d-d_F}}
$$
hence
\begin{equation} \label{eq64}
N_F(m, m^{\frac{1}{d-a_F}}) \ll_{n,d} m^{\frac{1}{d} + \frac{n-1}{d-d_F}}.
\end{equation}

By Lemma \ref{lem10} the integer points $\textbf{a}$ with $1 \leq |F(\textbf{a})| \leq m$ for which $\left\Vert \textbf a \right\Vert \geq m^{1/(d-a_F)}$ lie in a set of cardinality
$$
\ll_{n,d} \mathcal{H}(F)^{c_F}m^{\frac{n-1}{d-a_F}}(1 + \log m)^{n-2}
$$
or in
$$
\ll_{n,d} (1 + \log m + \log \mathcal{H}(F))^{n-1}
$$
proper subspaces. Thus, by \eqref{eq52} and \eqref{eq59}, these points lie in a set of cardinality
$$
\ll_{n,d} m^{\frac{n}{d}}
$$
or in
$$
\ll_{n,d} (1 + \log m)^{n-1}
$$
proper subspaces. Together with \eqref{eq64} we conclude that when \eqref{eq52} holds the integer points $\textbf{a}$ with $1 \leq |F(\textbf{a})| \leq m$ lie in a set of cardinality
\begin{equation} \label{eq65}
\ll_{n,d} m^{\frac{1}{d} + \frac{n-1}{d-d_F}}
\end{equation}
or in
\begin{equation} \label{eq66}
\ll_{n,d} (1 + \log m)^{n-1}
\end{equation}
proper subspaces.

Therefore, by \eqref{eq55}, \eqref{eq56}, \eqref{eq65} and \eqref{eq66} the integer points $\textbf{a}$ with $1 \leq |F(\textbf{a})| \leq m$ lie in a set of cardinality
\begin{equation} \label{eq67}
\ll_{n,d} m^{\frac{1}{d} + \frac{n-1}{d-d_F}}
\end{equation}
or in 
\begin{equation} \label{eq68}
\ll_{n,d} (1 + \log m)^{n-1}
\end{equation}
proper subspaces. Notice that we may suppose that each proper subspace is defined over $\mathbb{Q}$ by replacing the subspace by the linear span of the integer points in the subspace.

If $T$ is a proper rational subspace of $\mathbb{R}^n$ which contains a point with integer coordinates $\textbf{a}$ with $1 \leq |F(\textbf{a})| \leq m$ then $T$ is not contained in $A_1, ... ,A_k$. Since $F$ is of essentially finite type then $V(F_{|T})$ is finite, where $F_{|T}$ denotes $F$ restricted to $T$, and $V(\tilde F_{|T})$ is finite whenever $\tilde F_{|T}$ is $F_{|T}$ restricted to a rational subspace of $T$ which is not a subspace of $A_i$ for $i=1, ... ,k$. Therefore $F_{|T}$ is of essentially finite type.

We shall now prove our result by induction on $n$. First suppose that $n=2$. Then a proper rational subspace $T$ is defined by an equation of the form $pX_1 + qX_2 = 0$ with $p$ and $q$ coprime integers. The integer points in $T$ have the form $(kq,-kp)$ with $k$ an integer. If $T$ contains a point $\textbf{a}$ with $1 \leq |F(\textbf{a})| \leq m$ then it contains a point $\textbf{a}$ with $F(\textbf{a}) \neq 0$ and so $F(q,-p) \neq 0$. But $F(kq,-kp) = k^nF(q,-p)$ and if $1 \leq |F(kq,-kp)| \leq m$ then $1\leq |k|^n \leq m$ hence each proper rational subspace contains at most $2m^{\frac{1}{d}}$ points $\textbf{a}$ with $1 \leq |F(\textbf{a})| \leq m$ . Thus the result holds for $n=2$ by \eqref{eq67} and \eqref{eq68}.

$F_{|T}$ is of essentially finite type for each proper rational subspace $T$ which contains a point $\textbf{a}$ with integer coordinates for which $1 \leq |F(\textbf{a})| \leq m$. Let $n^{'}$ be the dimension of $T$, so $n^{'} < n$. As on page 801 of \cite{Th1}, $F_{|T}$ may be viewed as a form $G$, which is essentially finite, on $\mathbb{R}^{n^{'}}$. Thus, since $n^{'} \leq n-1$, by our inductive hypothesis each proper rational subspace of $\mathbb{R}^n$ which contains an integer point $\textbf{a}$ with $1 \leq |F(\textbf{a})| \leq m$ contains
$$
\ll_{n,d} m^{\frac{1}{d} + \frac{n-2}{d-d_F}}
$$
such points. Therefore, by \eqref{eq67} and \eqref{eq68}, the number of integer points $\textbf{a}$ with $1 \leq |F(\textbf{a})| \leq m$ is
$$
\ll_{n,d} m^{\frac{1}{d} + \frac{n-1}{d-d_F}} + (1+ \log m)^{n-1}m^{\frac{1}{d} + \frac{n-2}{d-d_F}}
$$
which is
$$
\ll_{n,d} m^{\frac{1}{d} + \frac{n-1}{d-d_F}} 
$$
as required.

\section{Proof of Theorem 2}

If $F$ is proportional to a power of a definite quadratic form the result holds by Lemma \ref{lem7} so we may assume this is not the case.

Let $N^{*}_F(m, m^{\frac{1}{d-a_F}})$ denote the number of integer points $\textbf{a}$ with $1 \leq |F(\textbf{a})| \leq m$ and $\left\Vert \textbf a \right\Vert \leq m^{1/(d-a_F)}$ and let $\tilde N^{*}_F(m, m^{\frac{1}{d-a_F}})$ denote the number of integer points $\textbf{a}$ with $1 \leq |F(\textbf{a})| \leq m$ and $\left\Vert \textbf a \right\Vert > m^{1/(d-a_F)}$. Then
$$
N^*_F(m) = N^{*}_F(m, m^{\frac{1}{d-a_F}}) + \tilde N^{*}_F(m, m^{\frac{1}{d-a_F}})
$$
and so 
\begin{equation} \label{eq69}
|V_Fm^{\frac{n}{d}} - N^*_F(m)| \leq |V_Fm^{\frac{n}{d}} - N^*_F(m, m^{\frac{1}{d-a_F}})| + \tilde N^{*}_F(m, m^{\frac{1}{d-a_F}}).
\end{equation}

We first estimate $|V_Fm^{\frac{n}{d}} - N^*_F(m, m^{\frac{1}{d-a_F}})|$, To this end we note that by Lemma \ref{lem8} with $B= m^{\frac{1}{d-a_F}}$
$$
|V_Fm^{\frac{n}{d}} - Vol(m, m^{\frac{1}{d-a_F}})| \ll_{n,d}  \mathcal{H}(F)^{c_F}m^{\frac{n-1}{d-a_F}}(1 + \log m)^{n-2}.
$$
It follows that
\begin{equation} \label{eq70}
|V_Fm^{\frac{n}{d}} - Vol^{'}(m, m^{\frac{1}{d-a_F}})| \ll_{n,d} \mathcal{H}(F)^{c_F}m^{\frac{n-1}{d-a_F}}(1 + \log m)^{n-2}.
\end{equation}
By Lemma \ref{lem9}, with $B=m^{\frac {1}{d-a_F}}$,
\begin{equation} \label{eq71}
| Vol^{'}(m, m^{\frac{1}{d-a_F}}) - N^{'}_F(m, m^{\frac{1}{d-a_F}})| \ll_{n,d} m^{\frac{n-1}{d-a_F}}.
\end{equation}
Denote by $N^{*'}_F(m, m^{\frac{1}{d-a_F}})$ the number of integer points $\textbf{a}$ with $1 \leq |F(\textbf{a})| \leq m$ and $max(|a_1|, ... ,|a_n|) \leq m^{\frac{1}{d-a_F}}$. The integer points $\textbf{a}$ for which $F(\textbf{a}) = 0$ lie in the subspaces $A_1, ... ,A_k$ with $k \leq d$ and each subspace contains at most $(1 + 2m^{\frac{1}{d-a_F}})^{n-1}$ points hence
\begin{equation} \label{eq72}
| N^{'}_F(m, m^{\frac{1}{d-a_F}}) - N^{*'}_F(m, m^{\frac{1}{d-a_F}})| \ll_{n,d} m^{\frac{n-1}{d-a_F}}
\end{equation}
Next we observe that
\begin{equation} \label{eq73}
| N^{*'}_F(m, m^{\frac{1}{d-a_F}}) - N^{*}_F(m, m^{\frac{1}{d-a_F}})| \leq  \tilde N^{*}_F(m, m^{\frac{1}{d-a_F}}).
\end{equation}
Thus, by \eqref{eq69}, \eqref{eq70}, \eqref{eq71}, \eqref{eq72} and \eqref{eq73},
\begin{equation} \label{eq74}
| V_Fm^{\frac{n}{d}} - N^{*}_F(m)| \ll_{n,d}  \mathcal{H}(F)^{c_F}m^{\frac{n-1}{d-a_F}}(1 + \log m)^{n-2} +  \tilde N^{*}_F(m, m^{\frac{1}{d-a_F}}).
\end{equation}

It remains to estimate $\tilde N^{*}_F(m, m^{\frac{1}{d-a_F}})$. By Lemma \ref{lem10} the integer points $\textbf{a}$ with $1 \leq |F(\textbf{a})| \leq m$ and $\left\Vert \textbf a \right\Vert \geq m^{1/(d-a_F)}$ are contained in a set of cardinality
\begin{equation} \label{eq75}
 \ll_{n,d}  \mathcal{H}(F)^{c_F}m^{\frac{n-1}{d-a_F}}(1 + \log m)^{n-2} 
\end{equation}
or lie in
\begin{equation} \label{eq76}
 \ll_{n,d}  (1 + \log m + \log \mathcal{H}(F))^{n-1} 
 \end{equation}
proper subspaces.

If $T$ is a proper rational subspace of $\mathbb{R}^n$ of dimension $k$ with $k < n$ which contains an integer point $\textbf{a}$ with $1 \leq |F(\textbf{a})| \leq m$ then, as we remarked in the proof of Theorem \ref{Theorem 1}, $F_{|T}$ is of essentially finite type and $F_{|T}$ may be viewed as a form in $k$ variables with $k \leq n-1$. Thus by Theorem \ref{Theorem 1} each proper subspace contains
\begin{equation} \label{eq77}
 \ll_{n,d}  m^{\frac{1}{d} + \frac{n-2}{d-d_F}} 
 \end{equation}
integer points  $\textbf{a}$ with $1 \leq |F(\textbf{a})| \leq m$ and $\left\Vert \textbf a \right\Vert \geq m^{1/(d-a_F)}$. Therefore, by \eqref{eq75}, \eqref{eq76} and \eqref{eq77},
\begin{equation} \label{eq78}
 \tilde N^{*}_F(m, m^{\frac{1}{d-a_F}})  \ll_{n,d} \mathcal{H}(F)^{c_F}m^{\frac{n-1}{d-a_F}}(1 + \log m)^{n-2}  + (1 + \log m + \log \mathcal{H}(F))^{n-1} 
m^{\frac{1}{d} + \frac{n-2}{d-d_F}}. 
 \end{equation}
Our result now follows from \eqref{eq74} and \eqref{eq78}.

\section{Acknowledgements}

 This research was supported in part by  grant A3528 from the Natural Sciences and Engineering Research Council of Canada.


\begin{thebibliography}{HD}










\normalsize
\baselineskip=17pt


\bibitem{B} M.A.~Bean, An isoperimetric inequality for the area of plane regions defined by binary forms, {\em Compositio Math.\/} {\bf 92} (1994), 115--131.

\bibitem{BT} M.A.~Bean and J.L.~Thunder, Isoperimetric inequalities for volumes associated with decomposable forms, {\em J.London Math. Soc.\/} {\bf 54} (1996), 39-49.



\bibitem{BC} Z.I.~Borevitch and I.R.~Chafarevitch, \textit{Th\'eorie des nombres},  Gauthier-Villars, Paris, 1967.

\bibitem{Daven} H.~Davenport, On a principle of Lipschitz, {\em J.London Math. Soc. \/} {\bf 26} (1954), 179-183.


\bibitem{Davis} P.~Davis, Gamma function and related functions, \textit{Handbook of Mathematical Functions}, (M.~Abramowitz and I.~Stegun eds.), Dover Publications, New York, 1965.


\bibitem{JH1} J.-H.~Evertse, The number of solutions of decomposable form equations, {\em Inventiones Math.\/} {\bf 122} (1995), 559-601.

\bibitem{JH2} J.-H.~Evertse, An improvement of the quantitative Subspace theorem, {\em Compositio Math.\/} {\bf 101} (1996), 225-311.


\bibitem{JH3} J.-H.~Evertse, On the norm form inequality $|F(x)| \leq h$, {\em Publicationes Math. Debrecen \/} {\bf 56} (2000), 337-374.

\bibitem{EG} J.-H.~Evertse and K.~Gy\H{o}ry,  \textit{Unit equations in Diophantine number theory}, Camb. Stud. Adv. Math. {\bf 146}, Cambridge University Press, 2015.


\bibitem{L} E.~Landau, Zur analytischen Zahlentheorie der definiten quadratische Formen (\"Uber die Gitterpunkte in einem mehrdimensionalen Ellipsoid),  {\em Sitzber. Preuss. Akad. Wiss.\/} {\bf 31} (1915) 458-476.

\bibitem{M} K.~Mahler, Zur Approximation algebraischer Zahlen III, (\"Uber die mittlere Anzahl der Darstellungen grossen Zahlen durch bin\"are Formen), {\em Acta Math.\/} {\bf 62} (1933), 91-166.



\bibitem{R} K.~Ramachandra, A lattice-point problem for norm forms in several variables, {\em Journal of Number Theory \/} {\bf 1} (1969), 534-555.



\bibitem{Sch1} W.M.~Schmidt, Norm form equations {\em Annals of Math.\/} {\bf 96} (1972), 526-551.

\bibitem{Sch2} W.M.~Schmidt, The subspace theorem in diophantine approximation, {\em Compositio Math.\/} {\bf 69} (1989), 121-173.

\bibitem{Sch3} W.M~Schmidt, The number of solutions of norm form equations, {\em Transactions of the A.M.S.\/}  {\bf 317} (1990), 197-227.



\bibitem{Sch4} W.M.~Schmidt, \textit{Diophantine approximations and Diophantine equations}, Lecture Notes in Mathematics, {\bf 1467}, Springer-Verlag, Berlin, Heidelberg, 1991.

\bibitem{St1} C.L.~Stewart and S.Y.~Xiao, On the representation of integers by binary forms, {\em Math. Annalen\/} {\bf 375} (2019), 133-163. 

\bibitem{St2} C.L.~Stewart and S.Y.~Xiao, On the representation of k-free integers by binary forms {\em Revista Mat. Iberoamericana \/} {\bf 37} (2021), 723-748.

\bibitem{Th1} J.L.~Thunder, Decomposable form inequalities {\em Annals of Math.\/} {\bf 153} (2001), 767-804.



\bibitem{Th2} J.L.~Thunder, Volumes and diophantine inequalities associated with decomposable forms, {\em Journal of Number Theory\/} {\bf 101} (2003), 294-309. 


\bibitem{Th3} J.L.~Thunder, Asymptotic estimates for the number of integer solutions to decomposable form inequalities, {\em Compositio Math.\/} {\bf 141} (2005), 271-292.




\end{thebibliography}
\end{document}